    \documentclass[reqno%,a4paper,makeidx
    ]{amsproc}
    \usepackage{amsmath,amssymb,amsthm,amsfonts,latexsym}
   \usepackage{euscript,mathrsfs}
   \usepackage[dvipsnames,usenames]{color}
     \usepackage{graphicx}
    \theoremstyle{plain}
   \newtheorem{thm}{Theorem}
     \newtheorem{pro}[thm]{Proposition}
   
   \newtheorem{cor}[thm]{Corollary}

    \newtheorem*{mas}{Quintessence}

   \theoremstyle{definition}
   
   \newtheorem{exa}[thm]{{\it Example}}

   \theoremstyle{remark}
   \newtheorem{rem}[thm]{{\it Remark}}

\DeclareMathOperator{\okr}{{\stackrel{{\scriptscriptstyle{\mathsf{def}}}}{=}}}
 \DeclareMathOperator{\D}{d } 
\DeclareMathOperator{\lin}{lin}
\DeclareMathOperator{\diag}{{\rm diag}}
\def\funk#1#2#3{#1\colon#2\to#3}
\def\funkc#1#2#3#4#5{#1\colon#2\ni#3\mapsto#4\in#5}
\def\Ge{\geqslant}
\def\is#1#2{\langle#1,#2\rangle}

\def\isb#1#2{\big\langle#1,#2\big\rangle}

\def\Le{\leqslant}
\def\sbar#1{\,\overline{ #1}}
\def\wynik{$\implies$}
\def\zb#1#2{\{{#1}\colon\ {#2}\}}
\newcommand*\ddc{\mathcal D}
\newcommand*\hhc{\mathcal H}
\newcommand*\kkc{\mathcal K}
\newcommand*\ccb{\mathbb C}
\newcommand*\nnb{\mathbb N}
\newcommand*\rrb{\mathbb R}
\newcommand*\iif{\mathfrak I}
\newcommand*\aasm{\boldsymbol a}
\newcommand*\bbsm{\boldsymbol b}
\newcommand*\eesm{\boldsymbol e}
\newcommand*\sssm{\boldsymbol s}
\newcommand*\ttsm{\boldsymbol t}
\newcommand{\matp}[1]{\begin{pmatrix}#1\end{pmatrix}}

\newcommand{\sbs}{\subseteq}
\newcommand{\dts}{,\dots,}

\newcommand{\CS}{\ccb[\nnb]}
\newcommand{\CSS}{\ccb[\nnb\times\nnb]}
\newcounter{dc}

   \usepackage{color,colortbl}
\begin{document}
%WAŻNA LINIJKA 
\ifx\JPicScale\undefined\def\JPicScale{1}\fi \unitlength\JPicScale mm

     %%%%% DELETE WHEN PAPER FINNISHED %%%%%%%
    \noindent\small%{\jobname,\ draft\ of\ {\today}}
     \vspace{2.0cm}
   %%%%%%%%%%%

   \title[The Sobolev moment problem  ]{The Sobolev moment problem  and Jordan dilations }
   \author[F.H. Szafraniec and M. Wojtylak]{Franciszek Hugon Szafraniec and Micha{\l} Wojtylak}
   \address{Instytut Matematyki, Uniwersytet Jagiello\'nski,
ul. {\L}ojasiewicza 6, 30348 Krak\'ow, Poland}
   \email{umszafra@cyf-kr.edu.pl}
    \address{Instytut Matematyki, Uniwersytet Jagiello\'nski,
ul. {\L}ojasiewicza 6, 30348 Krak\'ow, Poland}
   \email{michal.wojtylak@im.uj.edu.pl}

   \thanks{Supported by the National Science Center, decision No.
DEC-2013/11/B/ST1/03613.}

\thanks{Preliminaries of this research were reported by  FHSz at {\em The 7th Conference on Function Spaces, May 13-17, 2014, Southern Illinois University, Edwardsville, Illinois} and {\em International Workshop on { Operator Theory and Applications, IWOTA 2015}, July 6-10, 2015, Tbilisi Mathematical Union, Georgian National Academy of Sciences, Ivane Javakhishvili Tbilisi State University, Tbilisi, Georgia.}
}
   %\subjclass{Primary   ;  Secondary    }
   %\keywords{ }
   %\dedicatory{}
   %%%%%%%%%%%%%%%%%%
   \begin{abstract} Moment problems and orthogonal polynomials, both meant in a single real variable, belong to the oldest problems in Classical Analysis. They have been developing for over a century in  two parallel, mostly independent streams. During the last 20 years a rapid advancement  of polynomials orthogonal in Sobolev space has been noticed, see \cite{paco3} for an updated survey; their moment counterpart seems to be not paid enough attention as it deserves.  In this paper we intend to resume the theme of \cite{paco1}  and also \cite{paco2}, and open the door for further, deeper study of the problem.
   \end{abstract}
   \maketitle
   %%%%%%%%%%%%%
   %\tableofcontents
   %%%%%%%%%%%%%%%%%%%%%%%%%%%%%%%%%%%%%%%%%%%%%%%%%%%%%%%%%%%%%%%%%%%
%%% INDEX %%%%
%%%%%%%%%%%%%%%%%%%%%%%%%%% %%%%%%%%%%%%%%%%%%%%%%%%%%%%%%%%%%%%%%%%

%\printindex

%%%%%%%%%%%%%%%%%%%%%%%%%%%%%%%%%%%%%%%%%%%%%%%%%%%%%%%%%%%%%%%%%%%

   %%%%%%%%%%%%%%%%%%%%%%%%%%%%%%%%%%%%%%%%%%%%%%%%%%%%%%%%%%%%%%

  \subsection{Opening}\label{s1}   {Solving a moment problem means, roughly speaking, to find a  spectral representation of a given data. This is usually understood as a typical inverse spectral problem, also because of the traditional connection between moments and operators; the latter are the object of special concern in the present paper. }

   {For example,}
  given a sequence $(a(n))_{n=0}^{\infty}$ of numbers, the classical (named after Hamburger) moment problem asks under which conditions there is a nonnegative Borel measure $\mu$ on $\rrb$ such that
\begin{equation}\label{1.15.10}
a(n)=\int\nolimits_{\rrb}t^{n}\mu(\D t),\quad n=0,1,\ldots
\end{equation}  
This happens if and only if 
$$
\sum\nolimits_{k,l}a(k+l)\xi_{m}\sbar\xi{n}\Ge0,\quad \text{ $\xi_{n}=0$ but a finite number of $n$'s.}
$$
The above is referred to as positive definiteness of the matrix\,\footnote{\;This is pretty often considered in terms of its determinants.} $(a(k+l))_{k,l=0}^{\infty}$, which in turn can be viewed as positive definiteness of $n \to  {a(n)}$ considered
as a function on the \underbar{involution} semigroup\,\footnote{\;We use $\nnb$ for $\{0,1,\dots\}$.} $\nnb$.   
The important feature of the above is translational invariance
% $$
%\fbox{$a(m+k,b)=a(m,n+k)$ for all $k\in\nnb$}
%$$
 $$
a(m+k,b)=a(m,n+k) \text{ for all } k\in\nnb
$$
of the Hankel matrix $(a(m,n))_{m,n=0}^{\infty}$, here with the  notation $a(m,n)\okr\int\nolimits_{\rrb}t^{m+n}\mu(\D t)$. This  {gives an opportunity} to make use of harmonic analysis ideas on involution semigroups combined with operator dilation theory as successfully done in \cite{nagy}.

In analogy with the above,  {in the present paper we consider the following problem.} Given a bisequence\,\footnote{\;It can be considered alternatively as an infinite matrix; we abandon this point of view in order to avoid any confusion.} $(s(m,n))_{m,n=0}^{\infty}$ of numbers,   {we ask}  under which conditions there is a $2\times 2$ positive definite matrix $\boldsymbol{\mu}\okr(\mu_{i,j})_{i,j=0}^{1}$ of measures on  $\rrb$ such that
\begin{equation}\label{1.27.07}
s(m,n)=\sum\nolimits_{i,j=0}^{1}\int\nolimits_{\rrb}{t^{m}}{\raisebox{3.5pt}{$^{(i)}$}}{t^{n}}{\raisebox{3.5pt}{$^{(j)}$}}\mu_{i,j}(\D t),\quad m,n=0,1,\ldots.
\end{equation}
The superscript ${}^{(i)}$ designates the $i$-th derivative and positive definiteness of the $2\times 2$ matrix valued function $(\mu_{i,j}(\,\cdot\,))_{i,j=0}^{1}$ defined on Borel subsets of $\rrb$ means that  for every Borel $\rho$ the $2\times2$  {complex} matrix $(\mu_{i,j}(\rho))_{i,j=0}^{1}$ is  positive definite.  Call  {this question} {the {\em Sobolev moment problem} as the integral on the right hand side of \eqref{1.27.07} is of Sobolev type. A bisequence $s$ enjoying the property \eqref{1.27.07} will be called a {\em Sobolev moment sequence}. Let us remark that positive definiteness of the matrix $\boldsymbol{\mu}$ does not exclude some of its entry measures $\mu_{i,j}$ to be signed or complex, (cf. \cite{rudin} for complex measures)  though it forces the diagonal entries $\mu_{i,i}$, $i=0,1$ to be positive measures anyway. 

The bisequence  $(s(m,n))_{m,n=0}^{\infty}$ is \underbar{no} \underbar{longer} translationally invariant (in Hankel sense) unlike in the  case of the Hamburger moment problem. Another, not translationally invariant ``moment problem'' on the real line $\rrb$ is treated\,\footnote{\;This is a cosine moment problem, not a power one like those commonly considered, so it  ``continuous'' in a sense.} in \cite{nie}. There a dilation argument is exploited as well.

Our main tool in analysing the Sobolev moment problem will be the theory of positive definite  forms, in the sense already defined and considered in \cite{ark}. Namely, we say that 
$$
\ttsm:\mathbb{N}\times \mathbb{N} \times\mathbb{C}^2\times\mathbb{C}^2\to \mathbb{C}
$$
is a \underline{form} over $\mathbb{N}\times \mathbb{N} \times\mathbb{C}^2\times\mathbb{C}^2$ (informally  abbreviated below to $\mathbb{N}^2\times \mathbb{C}^4$) if for any $m,n\in\mathbb{N}$ the mapping 
$\ttsm(m,n,\cdot,-)$ is Hermitian linear\footnote{ {The equivalent term  `sesqui-linear' appears often in the literature.}}. We say that $\ttsm$ is positive definite if 
\begin{equation}\label{2.27.07}
\sum_{k,l}\ttsm(m_{k},m_{l},\aasm_{k},\aasm_{l})\Ge0\quad\text{for any finite choice of sequences $(m_{k})_{k}$ and $(\aasm_{k})_{k}$,}
\end{equation}

Our main result can be summarised as follows. 
\begin{mas}
Given $(s(m,n))_{m,n=0}^{\infty}$, satisfying 
 the  condition 
\begin{equation}\label{phiD}
s(m+3,n)-3s(m+2,n+1)+3s(m+1,n+2)-s(m,n+3)=0,\quad m,n\in\nnb.
\end{equation}
there exists a form $\sssm$ over $\mathbb{N}^2\times \mathbb{C}^4$
which is translation invariant, i.e. satisfies
\begin{equation}\label{4.17.10}
\sssm(m+k,n,\aasm,\bbsm)=\sssm(m,n+k,\aasm,\bbsm),\quad m,n,k=0,1,\ldots
\end{equation}
and is in the following relation with $s$
\begin{equation}\label{sands}
\begin{array}{lcl}
s(m,n)&=&\sssm(m,n,\eesm_0,\eesm_0)+m\sssm(m-1,n,\eesm_1,\eesm_0)\\
&+&n\sssm(m,n-1,\eesm_0,\eesm_1)+mn\sssm(m-1,n-1,\eesm_1,\eesm_1),\quad m,n\in\mathbb{N},
\end{array}
\end{equation}
 with 
 \begin{equation}\label{2.28.09}
\eesm_{ 0}\okr(1,0) \text{ and } \eesm_{ 1}\okr(0,1)
\end{equation}
standing for  the basic vectors in $\ccb^{2}$ and a convention that whenever a negative argument appears the corresponding entry is valued to $0$.
Furthermore, if $\sssm$ is additionally positive definite
 then $s$ is a Sobolev moment bisequence. 

Conversely, for every Sobolev moment bisequence 
the condition \eqref{phiD} holds,
the form 
\begin{equation}\label{formdef}
\sssm(m,n,\aasm,\bbsm)=\sum_{ij=0}^1 a_i\bar b_j \int t^{m+n} d \mu_{ij},\quad \aasm=(a_0,a_1),\ \bbsm=(b_0,b_1),
\end{equation}
satisfies \eqref{sands} and is translation invariant and positive definite.

\end{mas}

Formula \eqref{phiD} comes from \cite[condition (1.2a), p. 309]{hel}, consequently we will call it the {\em Helton condition}.  Its operator theoretic master  can be found in \cite{hel}, however the circumstances there are rather limited.

\begin{rem}
As we show later, every Sobolev moment sequence has the form
$$
s(m,n)=\is{(T+N)^{m}V(1,0)}{(T+N)^{n}V(1,0)}_{\kkc},\quad m,n=0,1,\ldots
$$
where $\funk V{\ccb^{2}}\mathcal{K}$ is an isometry, $T$ is selfadjoint and $N$ is a nilpotent operator in some Hilbert space $\mathcal{K}$, see equations \eqref{1.19.10} and \eqref{3.19.10}, which can be viewed as the  Jordan dilation of $s$. 
\end{rem}

%In the present paper the main tool will be dilation theory, applied to the form $\sssm$, which will give rise to some questions of operator theoretic nature.

\subsection{The Sobolev moment problem; necessary conditions}\label{s2}

 {First we argue that the conditions appearing in  Quintessence are necessary, i.e. we show the 'Conversely' part of it. Suppose that $(s(m,n))_{m,n=0}^{\infty}$ is a Sobolev moment bisequence. }
The integral representation \eqref{1.27.07}  suggests the  following decomposition (whenever a negative argument appears the corresponding entry is valued to $0$)
\begin{align}\label{2.26.07}\begin{split}
s(m,n) &= s_{0.0}(m,n) +ms_{0,1}(m - 1,n) \\&+ ns_{1,0}(m,n - 1) + (m - 1)(n - 1)s_{1,1}(m - 1,n - 1)
\end{split}
\end{align}
as well as the properties 
\begin{equation}\label{1.26.07}\begin{split}
{s_{0,0}(m,n)=s_{0,0}(m+n,0)=s_{0,0}(0,m+n)}\\
{s_{0,1}(m,n)=s_{0,1}(m+n,0)=s_{0,1}(0,m+n)}\\
{s_{1,0}(m,n)=s_{1,0}(m+n,0)=s_{1,0}(0,m+n)}\\
{s_{1,1}(m,n)=s_{1,1}(m+n,0)=s_{1,1}(0,m+n)}.
\end{split}
\end{equation}
% {Let us isolate the following condition on an arbitrary bisequence $(s(m,n))_{m,n=0}^\infty$:
If a bisequence $s$ satisfies the above properties \eqref{2.26.07} and \eqref{1.26.07}  then for $\aasm=(a_{0},a_{1}),\bbsm=(b_{0},b_{1})\in\ccb^{2}$ and $m,n\in\nnb$ consider
\begin{equation}\label{3.26.07}
{\sssm(m,n,\aasm,\bbsm)}{\okr} a_{a}\bar b_{0}s_{0,0}(m,n)+a_{0}\bar b_{1}s_{0,1}(m,n)+a_{1}\bar b_{0}s_{1,0}(m,n)+a_{1}\bar b_{1}s_{1,1}(m,n).
\end{equation}
Consequently $\sssm$ is a \underbar{form} over $\nnb^{2}\times\ccb^{4}$. %(also thought of as $(\nnb\times\ccb^{2})^{2}$ after permuting the variables) 
 Conditions \eqref{1.26.07} are equivalent to translational invariance of $\sssm$ on  $\nnb\times\nnb$, that is to
\begin{equation}\label{2.15.10}
\sssm(m,n,\aasm,\bbsm)=\sssm(m+n,0,\aasm,\bbsm)=\sssm(0,m+n,\aasm,\bbsm),\quad m,n=0,1,\ldots
\end{equation}
which is the same as \eqref{4.17.10}.  {The form $\sssm$ satisfies \eqref{sands}, which can be easily checked by direct calculation. To finish the study of the necessary conditions, we need to show the following statement.}

%Therefore the notion of positive definiteness of a form as done \cite{ark} applies. More precisely, 
%$\sssm$ is a {\em positive definite form} if
% \begin{equation}\label{2.27.07}
%\sum_{k,l}\sssm(m_{k},m_{l},\aasm_{k},\aasm_{l})\Ge0\quad\text{for any finite choice of sequences $(m_{k})_{k}$ and $(\aasm_{k})_{k}$.}
%\end{equation}
%
% {The following statement is valid.}

\begin{pro}\label{r1}
 If $(s(m,n))_{m,n=0}^{\infty}$  is a Sobolev moment bisequence and we define\footnote{ {Given a Sobolev moment sequence the decomposition into four translation invariant sequences \eqref{2.26.07}, \eqref{1.26.07} may not be unique, as will be seen later on, cf. Proof of Theorem  \ref{decomp}, formulas \eqref{Homo1} and \eqref{Homo2} therein. Therefore, in formula \eqref{5.19.10} we need to refer to a specific decomposition explicitly.}}
\begin{equation}\label{5.19.10}
s_{i,j}(m,n)\okr\int_{\rrb}t^{m+n}\mu_{i,j}(\D t),\quad i,j=0,1
\end{equation}
then $ \sssm$ defined by \eqref{3.26.07} is a positive definite form.
\end{pro}

Before proving it let us recall some necessary background. Note that if $s$ is a Sobolev moment sequence,  the measure $\boldsymbol{\mu}$ can be always viewed as a semispectral\, \footnote{\;It is sometimes referred to as POV(=positive operator valued) measure.} one in $\ccb^{2}$ (cf. \cite[Theorem 4 p. 30 ]{mlak}), i.e. a $2\times2$-matrix valued function 
defined on Borel subsets of $\rrb$, which is countably additive (each of the two possibilities: in strong and weak operator topology, in this case is equivalent in this case to the entry-wise convergence). A semispectral measure becomes \underbar{spectral} if it is orthogonal projection valued.  
 
 %The following result known as   Na\u{\i}mark dilation theorem turns out to be useful in the further developments.
\begin{thm}[cf. \cite{mlak}, p. 30]\label{t1.29.10}
A $2\times 2$  matrix $(\mu_{i,j})_{i,j=0}^{1}$ of measures on  $\rrb$ is positive definite if and only if there is a Hilbert space $\kkc$, and a bounded linear operator  $\funk R{\ccb^{2}}\kkc$ and a spectral measure $E$ in $\kkc$
such that
\begin{equation}\label{1.29.10}
(\mu_{i,j}(\rho))_{i,j=0}^{1}=R^{*}E(\rho)R,\quad \text{$\rho$ a Borel subset of $\rrb$}.
\end{equation}
\end{thm}
\begin{rem}\label{rxxx}
 In general $R^{*}R=(\mu_{i,j}(\rrb))_{i,j=0}^{\infty}$; in particular,  $R^{*}$ is the orthogonal projection of $\kkc$ onto $\ccb^{2}$ if and only if   $(\mu_{i,j}(\rrb))_{i,j=0}^{1}=\diag(1,1)$. 
\end{rem}
 
 Notice that \eqref{1.29.10} is equivalent to
$$
 \mu_{i,j}(\,\cdot\,)=\is{E(\,\cdot\,)R\eesm_{i}}{R\eesm_{j}},\quad i,j=0,1.
$$

\begin{proof}[Proof of Proposition \ref{r1}]

 {With $\aasm_k=(a_{0,k},a_{1,k})$ and $\kkc$, $E$ and $R$ as above, one has}
\begin{align*}
\sum_{k,l}\sssm(m_{k},m_{l},\aasm_{k},\aasm_{l})&=
\sum\nolimits_{k,l}\sum\nolimits_{i,j=0}^{1}\int_{\rrb}a_{i,k}{t^{m_k}}{\raisebox{3.5pt}{$^{(i)}$}}\sbar a_{j,l}{t^{m_l}}{\raisebox{3.5pt}{$^{(j)}$}}\mu_{i,j}(\D t)\\
&=\sum\nolimits_{i,j=0}^{1}\int_{\rrb}\big(\sum\nolimits_{k}a_{i,k}{t^{m_k}}{\raisebox{3.5pt}{$^{(i)}$}}\big)\big(\sum\nolimits_{l}\sbar a_{j,l}{t^{m_l}}{\raisebox{3.5pt}{$^{(j)}$}}\big)\mu_{i,j}(\D t)\\
&=\sum\nolimits_{i,j=0}^{1}\int_{\rrb}\big(\sum\nolimits_{k}a_{i,k}{t^{m_k}}{\raisebox{3.5pt}{$^{(i)}$}}\big)\big(\sum\nolimits_{n}\sbar a_{j,l}{t^{m_l}}{\raisebox{3.5pt}{$^{(j)}$}}\big)\is{E(\D t)R\eesm_{i}}{R\eesm_{j}}_{\kkc}\\
&=\sum\nolimits_{i,j=0}^{1}\isb{\int_{\rrb}\big(\sum\nolimits_{k}a_{i,k}{t^{m_k}}{\raisebox{3.5pt}{$^{(i)}$}}\big)\big(\sum\nolimits_{l}\sbar a_{j,l}{t^{m_l}}{\raisebox{3.5pt}{$^{(j)}$}}\big)E(\D t)R\eesm_{i}}{R\eesm_{j}}_{\kkc}\\
&\stackrel{{\tt s}}=\int_{\rrb}\big\|\sum\nolimits_{i=0}^{1}\sum\nolimits_{k}a_{i,k}{t^{m_k}}{\raisebox{3.5pt}{$^{(i)}$}}E(\D t)R\eesm_{i}\big\|^{2}_{\kkc}\Ge0,
\end{align*}
where in the passage $\stackrel{\tt s}{=}$ multiplicativity of the spectral integral is used.
\end{proof}

\subsection{More on  Helton's condition}

We prove one of our basic results. 

\begin{thm}\label{decomp} { Given a bisequence $(s(m,n))_{m,n=0}^{\infty}$, the following conditions are equivalent}
\begin{enumerate}
\item[(h1)] $s$ satisfies the Helton condition \eqref{phiD};
\item[(h2)] there are four bisequences $(s_{i,j}(m,n))_{m,n=0}^{\infty}$, $i,j=0,1$ with 
the translation invariance properties \eqref{1.26.07} and  such that the bisequence $s$ decomposes as in \eqref{2.26.07};
\item[(h3)] there exists a form $\sssm$ with 
the translation invariance property \eqref{4.17.10}, satisfying \eqref{sands}.
\end{enumerate}

%satisfies, p. \pageref{gw},
% can be represented as a sum 
% \begin{equation}\label{rep}
% \begin{gathered}
% s(m,n)=s_{0,0} (m+n) +  m s_{0,1}(m+n-1) \\
% + n s_{1,0}(m+n-1) + mn s_{1,1}s(m+n-2),\quad m,n\in\nnb
% \end{gathered}
% \end{equation}
% with some $s_{i,j}:\nnb\to\ccb$ $(i,j=0,1)$ 
%  if and only if \eqref{phiD} holds.
%\begin{equation}\label{phiD}
%s(m+3,n)-3s(m+2,n+1)+3s(m+1,n+2)-s(m,n+3)=0,\quad m,n\in\nnb.
%\end{equation}
% {Furthermore, if \eqref{g1.7.11} holds then the form $\sssm$ defined by \eqref{3.26.07} is in the  relation \eqref{sands} to the bisequence $s$.}
\end{thm}

\begin{proof}
The implication (h3)$\Rightarrow$(h2) is trivial, the reverse implication was proved in the previous section, see formula \eqref{3.26.07}.
To show (h2)$\Rightarrow$(h1) it is enough to evaluate  the left hand side of \eqref{phiD}, which after using  we use \eqref{2.26.07} and \eqref{1.26.07} and   grouping of the summands results in  getting $0$. 
The main part of the proof is the implication (h1)$\Rightarrow$(h2).

Assume that  (h1) is satisfied. Let $\ccb[\nnb]$ stands for the linear space of  
of complex finitely supported functions on $\nnb$ considered point-wisely and $\ccb[\nnb\times\nnb]$ that of finitely supported functions on $[\nnb\times\nnb]$. 
The Kronecker deltas $\delta_{n}$ and $\delta_{m,n}$   are apparently members of $\ccb[\nnb]$ and $\ccb[\nnb\times\nnb]$, respectively.
We naturally embed $\nnb$ in  $\CS$ as $n\mapsto\delta_n$; therefore  $a\in\CS$ can be written as
$\sum_{n\in \nnb} a(n)\delta_n$; for technical reasons  set $\delta_{-1}=\delta_{-2}=0$. 

The bisequence $s\okr(s(m,n))_{m,n=0}^{\infty}$  determines the unique linear mapping
$$
s^\dagger:\CSS\to\ccb
$$  
by 
$$
 s ^\dagger(a)\okr\sum_{m,n\in  \nnb} a(m,n) s (m,n),\quad a\in\CSS.
$$
Clearly
$$
 s (m,n)= %s ^\dagger(\delta_m,\delta_n)=
  s ^{\dagger}(\delta_{m,n}).
$$

Moreover if we consider $\ccb[\nnb]$  as  the commutative algebra with the standard convolution 
$$
(a*b)(n)=\sum_{x+y=n} a(x)b(y), \quad a,b\in\ccb[\nnb]
$$
the condition \eqref{phiD} turns out to be equivalent to
$$
s^\dagger (p_1*\delta_{m,n})=0,\quad m,n\in\nnb,
$$
with 
$$
p_1=\delta_{3,0}-3\delta_{2,1}+3\delta_{1,2}-\delta_{0,3}.
$$

% and a bilinear mapping 
%$$
% s ^{(\dagger)}:\CS\times\CS\to \ccb
%$$
%by
%$$
% s ^{(\dagger)}:= \sum_{m,n\in  \nnb} a(m) b(n) s (m,n),\quad a,b\in\CS.
%$$
% 
%Notice that if $\rho:\CS\times\CS\to \CSS$ is the bilinear mapping defined by
%$$
%\rho(\delta_m,\delta_n)=\delta_{m,n},\quad m,n\in  \nnb,
%$$
%then 
%\begin{equation}\label{k1}
% s ^{(\dagger)}= s ^\dagger\circ\rho.
%\end{equation}
%All this gives us the following diagram to commute.
%%\begin{figure}[htb]
%\begin{center}
%\begin{picture}(95,30)(0,0)
%\pu\sssm(10,10){\makebox(0,0)[cc]{$\CS\times\CS$}}
%\pu\sssm(50,10){\makebox(0,0)[cc]{$\CSS$}}
%\pu\sssm(85,10){\makebox(0,0)[cc]{$\ccb$}}
%
%\pu\sssm(20,10){\vector(1,0){20}}
%\pu\sssm(30,13){\makebox(0,0)[cc]{$\rho$}}
%
%{\thicklines
%\pu\sssm(11,15){\line(1,1){6}}
%\pu\sssm(17,21){\line(1,0){58}}
%\pu\sssm(75,21){\vector(1,-1){7}}}
%\pu\sssm(47,24){\makebox(0,0)[cc]{$ s ^{(\dagger)}$}}
%
%\pu\sssm(60,10){\vector(1,0){20}}
%\pu\sssm(68,13){\makebox(0,0)[cc]{$ s ^\dagger$}}
%
%\end{picture}
%\end{center}
%
%Let
% $$
%\nabla:\CS\to\ccb^2[ \nnb],\quad \nabla(p)=p',\ p\in\CS,
% $$
%where $p'$ is a (formal) derivative of $p$. 
%By $(\nabla,\nabla)$ we define the linear mapping 
%$$
%(\nabla,\nabla):\CS\times\CS\ni(a,b)\mapsto(\nabla(a),\nabla(b))\in\ccb^2[ \nnb]\times\ccb^2[ \nnb],
%$$
%and by $M$ we define the bilinear mapping 
%$$
%M:\ccb^2[ \nnb]\times\ccb^2[ \nnb]\to \ccb^{2\times 2}[ \nnb]
%$$
%$$
%M(a_0,a_1;b_0,b_1)=(a_i*b_j)_{ij=0}^1,\quad a_0,a_1,b_0,b_1\in\CS,
%$$
%where  $\ccb^{2\times 2}[ \nnb]$  is understood as the linear space of $2\times 2$ matrices with entries in $\CS$. 
We define a linear mapping
 $$
 \Gamma: \CSS \to \ccb^{2\times 2}[ \nnb],
 $$
by
$$
\Gamma(\delta_{m,n}) = \begin{pmatrix}
\delta_{m+n} & m \delta_{m+n-1}\\
n\delta_{m+n-1} & mn \delta_{m+n-2}
\end{pmatrix},\quad m,n\in\nnb.
$$
%and observe that
%\begin{equation}\label{k2}
%\Gamma\circ\rho = M\circ (\nabla,\nabla).
%\end{equation}
%Now we got another diagram, which also commutes.
%\begin{center}
%\begin{picture}(60,50)(0,0)
% 
%\pu\sssm(10,40){\makebox(0,0)[cc]{$\CS\times\CS$}}
%\pu\sssm(50,40){\makebox(0,0)[cc]{$\CSS$}}
%
%\pu\sssm(10,10){\makebox(0,0)[cc]{$\ccb^2[ \nnb]\!\times\!\ccb^2[ \nnb]$}}
%\pu\sssm(53,10){\makebox(0,0)[cc]{$\ccb^{2\times 2}[ \nnb\times  \nnb]$}}
%{\thicklines
%\pu\sssm(22,10){\vector(1,0){20}}
%\pu\sssm(32,13){\makebox(0,0)[cc]{$M$}}
%
%\pu\sssm(20,40){\vector(1,0){20}}
%\pu\sssm(30,43){\makebox(0,0)[cc]{$\rho$}}
%}
%
%
%\pu\sssm(10,35){\vector(0,-1){20}}
%\pu\sssm(12,25){\makebox(0,0)[lc]{$(\nabla,\nabla)$}}
%
%
%\pu\sssm(50,35){\vector(0,-1){20}}
%\pu\sssm(52,25){\makebox(0,0)[lc]{$\Gamma_\nabla$}}
%
%
%\end{picture}
%\end{center}

Let us now get back to the assumption \eqref{phiD}. Note that it can be rewritten as
\begin{equation}\label{Homo1}
\iif(p_1)\sbs \ker s^\dagger,
\end{equation}
where $\iif(p_1)$ denotes the ideal generated by $p_1$. We will show \eqref{Homo2} later on that
\begin{equation}\label{Homo2}
\iif(p_1)=\ker \Gamma,
\end{equation}
now observe that this will finish the proof. Indeed,  the fundamental homomorphism theorem together with \eqref{Homo1} and \eqref{Homo2} implies that there exists a linear mapping 
$$
\Phi:\ccb^{2\times 2}[\nnb]\to\ccb,
$$
satisfying
$$
\Phi\circ\Gamma = s^\dagger.
$$
Writing this equality on $\delta_{m,n}$ we get 
$$
s(m,n)=s^{\dagger}(\delta_{m,n})= \Phi \begin{pmatrix}
\delta_{m+n} & m \delta_{m+n-1}\\
n\delta_{m+n-1} & mn \delta_{m+n-2}
\end{pmatrix}
$$ $$
=\Phi\begin{pmatrix}
\delta_{m+n} & 0\\
0 & 0
\end{pmatrix}+
m \Phi\begin{pmatrix}
0 & \delta_{m+n-1}\\
0 & 0
\end{pmatrix}
+
n\Phi\begin{pmatrix}
0 & 0\\
 \delta_{m+n-1} & 0
\end{pmatrix}+
mn\Phi\begin{pmatrix}
0 & 0\\
0 & \delta_{m+n-2},
\end{pmatrix}
$$
which gives the desired representation \eqref{2.26.07}  {and the relation \eqref{1.26.07}.}

Hence, to finish the proof of Theorem \ref{decomp} it is enough to show that \eqref{Homo2} holds.
The inclusion $\sbs$ follows by direct computation. To prove \eqref{Homo2} it remains to show the converse
$\ker \Gamma\sbs \iif(p_1)$. Let $a\in \ker \Gamma$ which means  
\begin{equation}
\label{e1}
\begin{gathered}
\sum_{m,n} a(m,n)\ \delta_{m+n}=0,\quad 
\sum_{m,n}  a(m,n)\ m\ \delta_{m+n-1} =0,\\
\sum_{m,n} a(m,n)\ n\ \delta_{m+n-1}=0,\quad
\sum_{m,n} a(m,n)\ mn\ \delta_{m+n-2}=0,
\end{gathered}
\end{equation}
with the convention that $\delta_{-1}=\delta_{-2}=0$.

 We show now  that  \eqref{e1} implies  $a\in \iif(p_1)$. The following decomposition holds
$$
a=\sum_{n\in\nnb} a^{n},\qquad a^n\okr \sum_{i=0}^n a(i,n-i)\delta_{i,n-i},\ n\in \nnb
$$
with the first sum to be terminating. Note that each  $a^n$ $(n\in\nnb)$ enjoys  the properties \eqref{e1} as well. Hence, without any loss of generality we
may  assume that
 $$
  a=\sum_{i=0}^n a_i \delta_{i,n-i},
$$
for which formulae \eqref{e1} reduce to
$$
\sum_{i=0}^n a_i=0 ,\quad \sum_{i=0}^n i\cdot a_i=0,\quad \sum_{i=0}^n (n-i)\cdot a_i=0 ,\quad
\sum_{i=0}^n i(n-i)\cdot a_i=0.
$$
Observe, that the second of these equations is a consequence of the first and the third and we may skip it.  The remaining equations can be written in the equivalent form
$$
(a_0\dts a_n)^\top\in\ker X_n,
$$
where the matrix $X_n$ is given by
$$
X_n= \matp{ 1 & 1 &  1 & \cdots & 1 & 1\\
          %0 & 1 & 2& \cdots & n-1 & n\\
          n & n-1 & (n-2) & \cdots & 1 & 0 \\
          0\cdot n & 1\cdot(n-1) & 2\cdot(n-2) & \cdots & (n-1)\cdot 1 & n\cdot 0}.
$$
Note that for $n=0,1$ the matrix $X_n$ has a trivial kernel. For $n\geq 2$ the first and the last two columns of $X_n$ are clearly linearly independent and hence
\begin{equation}\label{dimker}
\dim\ker X_{n} = (n+1) -3 = n-2,\quad n\geq 3.
\end{equation}
We will show simultaneously by induction with respect to $n\in\nnb$, $n\geq 3$ that\\
\begin{itemize}
\item[(i)] if $(a_0\dts a_n)^\top\in\ker X_n$ is a solution of the above system of equations then
$\sum_{i=0}^n a_i\delta_{i,n-i}\in \iif(p_1)$,
\item[(ii)] there exists $(a_0\dts a_n)^\top\in\ker X_n$ with $a_0\neq 0$,
\end{itemize}
 which will finish the proof.
 The kernel of $X_3$ is spanned by
$(a_0,a_1,a_2,a_3)^\top=(-1,3,-3,1)^\top$, hence $a=p_1$ and in consequence (i) and (ii) are satisfied. Now suppose
that both (i) and (ii) are true for some $n\geq 3$. First observe that if $y\in \ker X_n$ than
$(0,y)^\top\in\ker X_{n+1}$. Indeed, subtracting the second row of $X_{n+1}$ from the third
one obtains the matrix
$$
 \matp{ 1 & 1 &  1 & \cdots & 1 & 1\\
          n+1 & n & (n-1) & \cdots & 1 & 0 \\
          -n-1 & 0  & (n-1) & \cdots & (n-1)\cdots 1 &  0}=\left( \begin{array}{c|c} 1 & \\ n+1
          & X_n \\ -n-1 \end{array}\right).
$$
Consequently, by \eqref{dimker}, there are $n-2$ linearly independent vectors in $\ker
X_{n+1}$ of the form $(0,a_1,a_2\dts a_{n+1})^\top$. Note that for all these vectors one has
$$
a=\sum_{i=1}^{n+1} a_i\delta_{i,n+1-i}= \delta_{1,0}*\sum_{i=0}^{n} a_{i+1}\delta_{i,n-i}.
$$
 By induction, the element
 $$
\sum_{i=0}^{n} a_{i+1}\delta_{i,n-i}
 $$
is in $\iif(p_1)$, as $(a_1\dts a_{n+1})^\top\in\ker X_n$. Hence, $a\in \iif(p_1)$. In view of
\eqref{dimker} (with $n$ replaced by $n+1$), to finish the induction step  it is enough to show that
there exists a vector $(a_0,a_1\dts a_{n+1})^\top\in\ker X_{n+1}$ with $a_0\neq 0$ and $a=\sum_{i=0}^{n+1}
a_i\delta_{i,n+1-i}\in \iif(p_1)$. By the induction assumption (ii) there exists $(b_0\dts
b_n)\in\ker X_n$ with $b_0\neq 0$. We define $b_{-1}=b_{n+1}=0$ and $a_i=b_{i-1}-b_{i}$,
$i=0\dts n+1$. Since $b=\sum_{i=1}^n b_i \delta_{i,n-i}\in \iif(p_1)$, we have
$$
a=\sum_{i=0}^{n+1} a_i \delta_{i,n+1-i}=p_0*b\in \iif(p_1), \quad p_0=\delta_{1,0}-\delta_{0,1}.
$$
Since $a_0=-b_0\neq0$, the  vector  $(a_0,a_1\dts a_{n+1})^\top$  is linearly independent from
vectors of the form $(0,y)^\top$ with $y\in\ker X_n$, which finishes the induction step.
\end{proof}

 {Comparing what we proved so far with Quintessence in Introduction we see that what  remains to  show is that the positive definiteness of $\sssm$ implies existence of the measures $(\mu_{ij})_{ij=0}^1$.  This turns in considering a moment problem of Hamburger type  for the  $2\times 2$ matrix $(\sssm(m,n,\aasm,\bbsm))_{m,n=0}^{\infty}$.  One of the ways of solving this problem is to employ  dilation theory,  as understood in  \cite{polar,ark}. This can  be conveniently carried out in the  RKHS environment. }

\subsection{The RKHS buildup}\label{s4}
As mentioned earlier positive definiteness \eqref{2.27.07} makes it possible to introduce the reproducing kernel Hilbert space and consider operators therein.
Define the sections $\sssm_{(n,\bbsm)}$ and their linear span $\ddc$ as
\begin{equation}\label{5.17.10}
\sssm_{(n,\bbsm)}\okr \sssm(\,\cdot\,,n,-,\bbsm),\quad\ddc\okr\lin\zb{\sssm_{(n,\bbsm)}}{(n,\bbsm)\in \nnb\times\ccb^{2}}
\end{equation}
and equipped $\ddc$ with an inner product extended from\,\footnote{\;Notice $\sssm_{(n,\bbsm)}$ is antilinear in $\bbsm$.}
\begin{equation}\label{2.03.10}
\is{\sssm_{(n,\bbsm)}}{\sssm_{(m,\aasm)}}\okr \sssm(m,n,\aasm,\bbsm).
\end{equation}

The reproducing kernel Hilbert space $\hhc$ determined by \eqref{2.03.10} is composed of complex valued functions on $\nnb\times\ccb^{2}$ whereas the reproducing property (restricted to $\ddc$ which is a dense subspace of $\hhc$) is precisely
\begin{equation}\label{3.17.10}
\sssm_{(n,\bbsm)}(m,\aasm)=\is{\sssm_{(n,\bbsm)}}{\sssm_{(m,\aasm)}},
\end{equation}
which is just the definition of the inner product read {\em \`a rebours}.

Fix $k\in \nnb$ and set for $(n,\bbsm)\in\nnb\times\ccb^{2}$
\begin{gather}
(\varPhi(k)\sssm_{(n,\bbsm)})(m,\aasm)\okr \sssm_{(n,\bbsm)}(k+m,\aasm),\quad (m,\aasm)\in\nnb\times\ccb^{2}\label{1.17.10}\\\varPsi(k)\sssm_{(n,\bbsm)}\okr \sssm_{(n+k,\bbsm)}\label{2.17.10}
\end{gather}
defined so far for $\sssm_{(n,\bbsm)}$'s.   Notice $\varPsi(k)$ extends linearly to the whole of $\ddc$ as an operator.

Fortunately, and exclusively in this case, 
\begin{align*}
\is{\sssm_{(m,\aasm)}}{\varPsi(k)\sssm_{(n,\bbsm)}}&\stackrel{\eqref{2.17.10}}{=}\is{\sssm_{(m,\aasm)}}{\sssm_{(n+k,\bbsm)}}\stackrel{\eqref{3.17.10}}{=}\sssm_{(m,\aasm)}(n+k,\bbsm)
\\
&\stackrel{\eqref{5.17.10}}{=}\sssm(n+k,m,\bbsm,\aasm)\stackrel{\eqref{4.17.10}}{=}\sssm(n,m+k,\bbsm,\aasm)
\\
&%\stackrel{\eqref{2.03.10}}{=}\is{sssm_{(m+k,\aasm)}}{sssm_{(n,\bbsm)}}
\stackrel{\eqref{5.17.10}}{=}\sssm_{(m+k,\aasm)}(n,\bbsm)
\stackrel{\eqref{1.17.10}}{=}\varPsi(k)\sssm_{m,\aasm}(n,\bbsm)
\\
&\stackrel{\eqref{3.17.10}}{=}\is{\varPsi(k)\sssm_{(m,\aasm)}}{\sssm_{(n,\bbsm)}},
\end{align*}
which implies $\varPsi(k)$ is symmetric. Furthermore
\begin{align}\label{1.18.10}
\varPhi(k)\sssm_{(n,\bbsm)}(m,\aasm)&\stackrel{\eqref{2.17.10}}{=}\sssm_{(n,\bbsm)}(k+m,\aasm)\stackrel{\eqref{5.17.10}}{=}\sssm_{(n+k,\bbsm)}(m,\aasm)
\stackrel{\eqref{2.17.10}}{=}\varPsi(k)\sssm_{(n,\bbsm)}(m,\aasm).
\end{align}
which means $\varPhi(k)=\varPsi(k)$ on $\ddc$. As a consequence  $\varPhi(k)$ extends to a linear operator as well. 

Both, $\varPhi(k)$
and $\varPsi(k)$ are the standard, naturally defined operators in a RKHS.

\subsection{Jordan operators in the RKHS $\hhc$}\label{s5}
With notation \eqref{2.28.09} in mind consider an operator $N$ defined on $\ddc$ as the linear extension of the formula
$$N\sssm_{(n,\bbsm)}\okr \sbar b_{0}\sssm_{(n,\eesm_{1}).}$$

Set $T\okr\varPhi(1)$; because the mapping $k\to\varPhi(k)$ is a semigroup homorphism on $\nnb$ to linear operators on $\ddc$ we have $T^{k}=\varPhi(k)$. 
Apparently,
\begin{equation}\label{3.18.10}
N\text{ commutes with } T \text{ on }\ddc
\end{equation}
and 
\begin{equation}\label{4.18.10}
N^{2}=0 \text{ on } \ddc, \text{ that  is }N\text{ is nilpotent}.
\end{equation}
Due to \eqref{3.18.10} and \eqref{4.18.10}, the (commutative) Newton's binomial formula is applicable and reduces to
\begin{equation}\label{1.04.10}
(T+N)^{k}=T^{k}+k\,T^{k-1}N\text{ on }\ddc.
\end{equation}

This motivates the following state of affair. Given an arbitrary Hilbert space $\kkc$, say\,\footnote{\;Notice now the space may be not be the same as the above hence the notation is different} call an  operator of the form $T+N$ where $T$ is  symmetric and $N$  is such that $N^{2}=0$, and $T$ and $N$ commute on a dense subspace $\ddc$ of $\kkc$, invariant for both and being a core of each of them, a {\em Jordan operator of order $2$};  the case when both of $T$ and $N$ are bounded is considered in \cite{hel}. Furthermore, call the couple $T$ and $N$ a {\em Jordan dilation of} $s$ {\em relative to the decomposition} \eqref{2.26.07} if 
\begin{equation}\label{1.19.10}
s(m,n)=\is{(T+N)^{m}V(1,0)}{(T+N)^{n}V(1,0)}_{\kkc},\quad m,n=0,1,\ldots
\end{equation}
where $\funk V{\ccb^{2}}\ddc$ is an isometry, and
\begin{equation}\label{3.19.10}
\sssm(l,k,\bbsm,\aasm)=\is{T^{k}V\bbsm}{T^{l}V\aasm},\quad k,l\in\nnb,\;\aasm,\bbsm\in\ccb^{2}.
\end{equation}

%
%\begin{equation}\label{6.17.10}
%\is{\varPhi(k)sssm_{(n,\bbsm)}}{sssm_{(m,\aasm)}}=\sssm(n+k,m,\bbsm,\aasm){=}\sssm(n,m+k,\bbsm,\aasm)
%=\is{sssm_{(n,\bbsm)}}{\varPsi(k)sssm_{(m,\aasm)}}
%\end{equation}
%Because $\varPsi(k)$ extends linearly to the whole of $\ddc$ this implies  
%that $\varPhi(k)$ is well defined and  
% ${\varPhi(k)}\subset\varPsi(k)^{*}$.
%

\subsection{The basic result}\label{s6}

\begin{thm}\label{t2.27.07}
Given $(s(m,n))_{m,n=0}^{\infty}$. The following conditions are equivalent:
\begin{enumerate}
\item
[(i)] the condition {\rm (h1)} is satisfied and  $\sssm$ defined by \eqref{3.26.07}  {and \eqref{5.19.10}} is a positive definite form;
\item
[(ii)]  the condition {\rm (h1)} is satisfied and there is a Jordan dilation $T$, $N$ of the bisequence $s$ relative to the decomposition \eqref{2.26.07};
\item
[(iii)] $s$ is a Sobolev moment bisequence. 
\end{enumerate}

\end{thm}

\begin{proof}
Most of the arguments have been already presented in the constructions above. Let us put them together.

(i)\wynik(ii). Use the constructions done in sections \ref{s4} and \ref{s5} maintaining the notation applied there. 
Consider the standard embedding
$$\funkc V{\ccb^{2}}\aasm{\sssm_{0,\aasm}}\hhc.
$$  
It is a matter of straightforward though lengthy verification that
\begin{equation}\label{1.26.09}
{s(m,n)}=\is{(T+N)^{m}V(1,0)}{(T+N)^{n}V(1,0)}_{\hhc},\quad m,n=0,1\ldots
\end{equation}
which is nothing but \eqref{1.19.10} under the specific circumstances considered there.

\noindent Indeed, due to \eqref{1.04.10}, \eqref{2.26.07} and \eqref{2.03.10},
\begin{multline*}
\is{(T+N)^{m}V(1,0)}{(T+N)^{n}V(1,0)}
\\
=\is{(T^{m}+mT^{m-1}N)V(1,0)}{(T^{m}+mT^{m-1}N)V(1,0)}
\\
=\is{T^{m}\sssm_{(0,\eesm_{0})}}{T^{n}\sssm_{(0,\eesm_{0})}}+n\is{T^{m}\sssm_{(0,\eesm_{0})}}{T^{n-1}\sssm_{(0,\eesm_{1})}}+m\is{T^{m-1}\sssm_{(0,\eesm_{0})}}{T^{n}\sssm_{(0,\eesm_{1})}}
\\
+mn\is{T^{m-1}\sssm_{0,\eesm_{1}}}{T^{n-1}\sssm_{0,\eesm_{1}}}
\\
=\is{\sssm_{(m,\eesm_{0})}}{\sssm_{(n,\eesm_{0})}}+n\is{\sssm_{(m,\eesm_{0})}}{\sssm_{(n-1,\eesm_{1})}}+m\is{\sssm_{(n-1,\eesm_{0})}}{\sssm_{(0,\eesm_{1})}}
+mn\is{\sssm_{m-1,\eesm_{1}}}{\sssm_{n-1,\eesm_{1}}}
\\
=\sssm(n,m,\eesm_{0},\eesm_{0})+n\sssm(n-1,m,\eesm_{1},\eesm_{0})+m\sssm(n,m-1,\eesm_{0},\eesm_{1})+mn\sssm(n-1,m-1,\eesm_{1},\eesm_{1})\\
=s_{0,0}(m,n)+ms_{0,1}(m-1,n)+ns_{1,0}(m,n-1)+mns_{1,1}(m-1,n-1)
=s(m,n).
\end{multline*}
%\begin{rem}\label{r1}
%The formula \eqref{1.26.09} becomes our basic \underbar{dilation} statement.
%\end{rem}

(ii)\wynik(iii).
%\begin{align*}
%\!\!\!\!\!\is{(T^{m}+mNT^{m})V\bbsm}{(T^{n}+nNT^{n})V\aasm}&=%,\quad m,n=0,1\ldots
%\is{(T^{m}+mNT^{m})sssm_{(0,\bbsm)}}{(T^{n}+nNT^{n})sssm_{(0,\aasm)}}
%\\&=\is{sssm_{m,\bbsm}+b_{1}sssm_{m-1,\eesm_{1}}}{sssm_{m,\aasm}+a_{1}sssm_{n-1,\eesm_{1}}} 
%\end{align*}
Because $V^{*}$ is a projection onto a $2$-dimensional space $T$ can be represented by $2\times 2$-matrix valued semispectral measure $\boldsymbol{\mu}$ (which is a compression\,\footnote{\;This is when the formula \eqref{1.29.10} is read ``right-to-left'' with $R$ being an isometry.} to the $2$-dimensional space of a spectral measure of \underbar{any} selfadjoint extension\,\footnote{\;Selfadjoint extensions may be quite diverse, often  not necessarily unitary equivalent, cf. \cite{naj} and also \cite{sz_naj}. In our case apparently they can be obtained as von Neumann extensions, as those acting in the same space, and also as Na\u \i mark extension, which go beyond the space. The latter applies also to the case of equal deficiency indices. In this way we may get plenty of representing measures $\mathbf{\mu}$,  cf. \cite{naj,sz_naj}. } of $T$) as 
\begin{equation}\label{3.04.10}
\sssm(l,k,\bbsm,\aasm)=\is{T^{k}V\bbsm}{T^{l}V\aasm}=\int\nolimits_{\rrb}t^{k+l}\is{\boldsymbol{\mu}(\D t)\bbsm}{\aasm}.
\end{equation}  
%More precisely,  $T$ is a symmetric operator and as such it may be either essentially selfadjoint, which results in existence of a unique spectral measure in  $\hhc$ or it may have a plenty of semispectral measures (see \cite{mlak} for their meaning) coming from diverse selfadjoint extensions, }.

Because $V\eesm_{i}\perp V\eesm_{j}$ ($V$ is an isometry!) if ${ i}\neq{j}$ and $\lin(\sssm_{(n,\bbsm)})_{(n,\bbsm)}$ is invariant for $T$,    we have the decomposition
\begin{equation}\label{4.04.10}
\begin{pmatrix}\; \mu_{0,0}(\,\cdot\,)&\mu_{0,1}(\,\cdot\,)\;\\
\;\mu_{1,0}(\,\cdot\,)&\mu_{1,1}(\,\cdot\,)\;
\end{pmatrix}
\okr
\begin{pmatrix}\; \is{\boldsymbol{\mu}(\,\cdot\,){\eesm_{ 0}}}{{\eesm_{ 0}}}&\is{\boldsymbol{\mu}(\,\cdot\,){\eesm_{0}}}{{\eesm_{1}}}\;\\
\;\is{\boldsymbol{\mu}(\,\cdot\,){\eesm_{1}}}{{\eesm_{0}}}&\is{\mu(\,\cdot\,){\eesm_{1}}}{{\eesm_{1}}}\;
\end{pmatrix}
\end{equation}
Again, with some effort,
\begin{multline*}
\is{(T+N)^{m}V(1,0)}{(T+N)^{n}V(1,0)}
\\
=\is{T^{m}\sssm_{(0,\eesm_{0})}}{T^{n}\sssm_{(0,\eesm_{0})}}+n\is{T^{m}\sssm_{(0,\eesm_{0})}}{T^{n-1}\sssm_{(0,\eesm_{1})}}+m\is{T^{m-1}\sssm_{(0,\eesm_{0})}}{T^{n}\sssm_{(0,\eesm_{1})}}
\\
+mn\is{T^{m-1}\sssm_{0,\eesm_{1}}}{T^{n-1}\sssm_{0,\eesm_{1}}}
\\
=\sssm(n,m,\eesm_{0},\eesm_{0})+n\sssm(n-1,m,\eesm_{1},\eesm_{0})+m\sssm(m-1,n,\eesm_{0},\eesm_{1})+mn\sssm(m-1,n-1,\eesm_{1},\eesm_{1})
\\
\stackrel{\eqref{3.04.10}\eqref{4.04.10}}{=}\int_{\rrb}t^{m}t^{n}\mu_{0,0}(\D t)+n\int _{\rrb}t^{m}t^{n-1}\mu_{0,1}(\D t)+m\int _{\rrb}t^{m-1}t^{n}\mu_{1,0}(\D t)
\\
+mn\int_{\rrb}t^{m-1}t^{n-1}\mu_{1,1}(\D t)
\end{multline*}
%$$\is{(T+N)^{m}V(1,1)}{(T+N)^{n}V(1,1)}={\sum\nolimits_{i,j=0}^{1}\int\nolimits_{\rrb}{t^{m}}{\raisebox{3.5pt}{$^{(i)}$}}{t^{n}}{\raisebox{3.5pt}{$^{(j)}$}}
%\mu_{i,j}(\D t)}
%$$
which confronted with \eqref{1.26.09} establishes
$${s(m,n)}={
\sum\nolimits_{i,j=0}^{1}\int\nolimits_{\rrb}{t^{m}}{\raisebox{3.5pt}{$^{(i)}$}}{t^{n}}{\raisebox{3.5pt}{$^{(j)}$}}
\mu_{i,j}(\D t),\quad m,n=0,1,\ldots}$$
(iii)\wynik(i). This statement was proved as Proposition \ref{r1}.
%This is the easiest implication if one takes into account \eqref{5.19.10}. The integral representation \eqref{5.19.10} generates the decomposition \eqref{2.26.07}, hence makes the definition \eqref{3.26.07} possible. This, in turn, settles (i).
\end{proof}

\subsection{Compactly supported representing measures $\boldsymbol{\mu}$}\label{scomp}
For a measure $\boldsymbol \mu$ such that \eqref{3.04.10} holds to be compactly supported it is necessary and sufficient the operator $T$ therein to be bounded. For that there is a number of equivalent conditions  \underbar{in} \underbar{terms} \underbar{of} \underbar{$\sssm$}}, which are applicable here due to the fact that $\sssm$ is a translationally invariant form in the sense of \eqref{2.15.10} or \eqref{4.17.10}. The conditions are listed in \cite[Lemma 2]{wuj2} and are originated in \cite{bull,pams,ark}. Let us pick up some of them explicitly.
\begin{pro}
The following conditions are equivalent to the boundedness of $T$
\begin{enumerate}
\item[$\boldsymbol\bullet$]
for every $k=0,1,\ldots$ there is $d(k)$ such that
$$
\sssm(m+k,m+k,\aasm,\aasm)\Le d(k)\sssm(m,m,\aasm,\aasm),\;m\in\nnb,\,\aasm\in\ccb^{2};
$$
\item
[$\boldsymbol\bullet$]
for every $k=0,1,\ldots$ the is a function $\alpha\colon\nnb\to[0,+\infty)$ such that $\alpha(k+l)\Le\alpha(k)\alpha(l)$, $k,l=0,1,\dots$ for which
$$
|\sssm(m,m,\aasm,\aasm)|\Le c(\aasm)\alpha(m),\;m\in\nnb,\,\aasm\in\ccb^{2};
$$
\item
[$\boldsymbol\bullet$]
For any $k=0,1,\ldots$ and a finite choice of $m_{i}$'s and $\aasm_{i}$'s
$$
\liminf_{l\to +\infty} \sum_{i,j}\sssm(m_{i}+lk,m_{j}+kl,\aasm_{i},\aasm_{j})^{2^{-l}}<+\infty
$$
\end{enumerate}
\end{pro}

\subsection{Defect indices of $T$.}

The following statement is true.

\begin{thm}
The operator $T$, defined as in Section \ref{s5}, has equal defect indices, not larger than 2.
\end{thm}

\begin{proof}
We define a conjugation (anti-linear) operator $J$ on $\mathcal{D}$ by
$$
J\left(\sum_{n,\aasm} \alpha_{n,\aasm} \sssm_{n,\aasm}\right)=\sum_{n,\aasm} \sbar\alpha_{n,\aasm} \sssm_{n,\aasm}.
$$
A standard reproducing kernel Hilbert space argument shows that $J$ is properly defined, preserves the norm and extends uniquely to the whole space $\mathcal{H}$. Furthermore, $JT=TJ$ and hence, by the von Neumann theorem,
the defect indices of $T$ are equal. 

Now let us take any element  $f\in\mathcal{N}(T^*-z)$, with $\text{Im} z>0$. Note that 
$$
\sssm_{(n,\aasm)}=\sbar a_0 \sssm_{(n,\eesm_0)} + \sbar a_1 \sssm_{(n,\eesm_1)}, \quad \aasm=(a_0,a_1),
$$
which can be easily checked by taking the inner product of the both sides with an arbitrary kernel function $s_{(m,b)}$ and using the fact that $\sssm$ is a form. Consequently, 
\begin{eqnarray*}
f(n,\aasm)&=&\is{ f}{\sssm_{(n,\aasm)}} = \is{ f}{T^n \sssm_{(0,\aasm)}}=\is{ \bar z ^n f}{ \sssm_{(0,\aasm)}}\\
&=&a_0 \is{ \bar z ^n f}{ \sssm_{(0,\eesm_0)}}+ a_1 \is{ \bar z ^n f}{ \sssm_{(0,\eesm_1)}}\\
&=& a_0\bar z ^n   f(0,\eesm_0)+ a_1\bar z ^n f(0,\eesm_1).
\end{eqnarray*}
This means that the kernel $\mathcal{N}(T^*-z)$ is of dimension at most two.
\end{proof}

We present now an example, where the defect indices of $T$ are indeed $(2,2)$. 

\begin{exa}
Consider an indeterminate measure $\mu$ on the real line, set $\mu_{i,j}=\delta_{i,j} \mu$ and define the form $\sssm$ by \eqref{formdef}. Then, 
$$
\is{\sssm_{(m,\eesm_0)}}{\sssm_{(n,\eesm_1)}}=\sssm(m,n,\eesm_0,\eesm_1)=\int t^{m+n} d\mu_{0,1}=0,\quad m,n\in\mathbb{N}.
$$
Hence, the RKHS $\mathcal{H}$ decomposes naturally as a direct sum $\mathcal{H}_0\oplus\mathcal{H}_1$ and the operator $T$ is accordingly a direct sum of two operators, both having defect indices $(1,1)$, due to the indeterminacy of $\mu$. Hence, the defect indices of $T$ are in this example precisely $(2,2)$.
\end{exa}

\subsection{Boundedness of the operator $N$.}

\begin{thm}
Suppose one of the equivalent conditions of Theorem \ref{t2.27.07} holds. Then the operator $N$ is bounded if and only if for some $\alpha>0$ the kernel  
\begin{equation}\label{piatek}
\begin{gathered}
{\ttsm_\alpha(m,n,\aasm,\bbsm)}{\okr}  a_{0}\bar b_{0}(\alpha s_{0,0}(m,n) -  s_{1,1}(m,n))\\
+\alpha a_{0}\bar b_{1}s_{0,1}(m,n)
+\alpha a_{1}\bar b_{0}s_{0,1}(m,n)+\alpha a_{1}\bar b_{1}s_{1,1}(m,n)
\end{gathered}
\end{equation}
is a positive definite form on $\nnb^{2}\times\ccb^{4}$, cf. \eqref{2.27.07}. 
\end{thm}

\begin{proof}
The operator $N$ is bounded if and only if
$$
\langle Nf, Nf \rangle \leq \alpha \langle f,f\rangle, \quad f=\sum_k  \sssm_{(m_k),\aasm_k}\in\mathcal D
$$
for some $\alpha>0$. This in turn is equivalent to 
$$
\alpha\sum_{k,l} \langle \sssm_{(m_k,\aasm_k)} , \sssm_{(m_l,\aasm_l)} \rangle - \sum_{k,l} a_{0,k} \bar a_{0,l} \langle  \sssm_{(m_k,e_1)} \sssm_{(m_l,e_1)}  \rangle \geq 0
$$
and the left hand side of the above inequality can be written as
$$
\sum_{k,l} \left( \alpha \sssm({m_l,m_k,\aasm_l,\aasm_k}) - a_{0,k} \bar a_{0,l}  \sssm({m_l,m_k,e_1,e_1})\right)=\sum_{k,l} \ttsm_\alpha({m_l,m_k,\aasm_l,\aasm_k}). 
$$
\end{proof}
\begin{rem}\label{t2.04.02}
If 
$$
s_{0,0}(m,n)=s_{1,1}(m,n) \text{ and }s_{0,1}(m,n)=s_{1,0}(m,n)=0,\; m,n=0,1,\ldots
$$
then the form defined in \eqref{piatek} with $\alpha=1$ is obviously positive definite. Hence $N$ is always bounded in this case; confront this with Example \ref{t3.04.02}.
\end{rem}

\begin{cor}\label{t1.04.02}
A necessary condition for boundedness of  $N$ is that 
\begin{equation}\label{poniedzialek1}
\sup_{n\geq 0} \frac{\int t^{2n} d\mu_{11} }{\int t^{2n} d\mu_{00}} < +\infty.
\end{equation}
\end{cor}
\begin{proof}
Indeed, note that
\begin{equation}\label{poniedzialek2}
\| N \sssm_{(n,\aasm)} \|^2= |a_0|^2 \| \sssm_{\eesm_1,m} \|^2 = |a_0|^2  \int t^{2m} \D\mu_{11}
\end{equation}
and
\begin{equation}\label{poniedzialek3}
\| \sssm_{(n,\aasm)} \|^2 = |a_0|^2 \int t^{2m} \D \mu_{00} + 2 \text{Re}\left( a_0 \bar a_1 \int t^{2m}\D \mu_{01}\right)+  |a_1|^2  \int t^{2m} \D\mu_{11}.
\end{equation}
Hence, if $N$ is bounded then 
%$$
%\| N \sssm_{(n,\aasm)} \|^2 \leq C \| \sssm_{(n,\aasm)} \|^2\quad \aasm\in\mathbb{C}^2,\ n\geq 0
%$$
%for some $C>0$.
 setting $a_1=0$ we get by \eqref{poniedzialek2} and \eqref{poniedzialek3} that \eqref{poniedzialek1} is satisfied.

\end{proof}

We present an example, showing that $N$ is not automatically bounded even in the case when the off-diagonal measures $\mu_{01}$ and $\mu_{10}$ are zero. 

\begin{exa}\label{t3.04.02}   
Let $\mu_{01}$ and $\mu_{10}$ be zero measures and let $\D\mu_{11}(t)=t^{2k} \D \mu_{00}(t)$ with some $k\geq 0$.
Then it is a matter of straightforward verification that $\sssm_{m,\bbsm}\in\mathcal D(N^*)$ for all $m\in\mathbb N$ and $\bbsm\in\mathbb C^2$ and
$$
N^* \sssm_{(m,\bbsm)} = b_1 \sssm_{(e_0,m+2k)}.
$$
If now $\mu_{00}$ is the Gaussian measure with variance one, then
$$
\frac{\int t^{2n} \D \mu_{11}}{ \int t^{2n} \D \mu_{11}}= \frac{\int t^{2n+2k} \D \mu_{00}}{ \int t^{2n} \D \mu_{00}}=\frac{(2n+2k-1)!!}{(2n-1)!!}\to\infty\quad (n\to\infty).
$$
Hence, the necessary condition \eqref{poniedzialek1} is violated and consequently $N$ must not  be bounded\,\footnote{\;Unbounded nilpotent appear in \cite{ota}.}, although $\mathcal D\subseteq\mathcal D({N^*})$. 
\end{exa}

The above, when compared with what is in Section \ref{scomp}, shows that boundedness of $T$ seem to have very little in common with that of $N$.

\subsection{Positive definiteness once more -- an open question}
First let us note the following fact.

\begin{pro}\label{t1.27.07}
If the bisequence $(s(m,n))_{m,n=0}^{\infty}$ and the form $\sssm$ are in the relation \eqref{sands} and $\sssm$ is a positive definite form, then  $s$  is positive definite, i.e.
\begin{equation}\label{1.28.09}
\sum\nolimits_{m,n}\lambda_{m}\sbar \lambda_{n}s(m,n)\Ge0,\quad (\lambda_{m,n})_{n,n} \text{ of finite length.}
\end{equation}
\end{pro}
%Positive definiteness \eqref{1.28.09}
% is the same as positive definiteness of the scalar valued kernel $s(m,n)$, $m,n$ defined on $\nnb\times\nnb$, cf. \cite{ark,wuj,wuj1}.
%Though \eqref{1.28.09} looks very innocent, it needs to be provided with some argument.
%\begin{proof}
%Let
%  Let $\kkc$, $E$ and $R$ be as Theorem \ref{t1.29.10}. Consequently,
%\begin{align*}
%\sum\nolimits_{m,n}\lambda_{m}\sbar \lambda_{n}s(m,n)&=
%\sum\nolimits_{m,n}\sum\nolimits_{i,j=0}^{1}\int_{\rrb}\lambda_{m}{t^{m}}{\raisebox{3.5pt}{$^{(i)}$}}\sbar\lambda_{n}{t^{n}}{\raisebox{3.5pt}{$^{(j)}$}}\mu_{i,j}(\D t)\\
%&=\sum\nolimits_{i,j=0}^{1}\int_{\rrb}\big(\sum\nolimits_{m}\lambda_{m}{t^{m}}{\raisebox{3.5pt}{$^{(i)}$}}\big)\big(\sum\nolimits_{n}\sbar\lambda_{n}{t^{n}}{\raisebox{3.5pt}{$^{(j)}$}}\big)\mu_{i,j}(\D t)\\
%&=\sum\nolimits_{i,j=0}^{1}\int_{\rrb}\big(\sum\nolimits_{m}\lambda_{m}{t^{m}}{\raisebox{3.5pt}{$^{(i)}$}}\big)\big(\sum\nolimits_{n}\sbar\lambda_{n}{t^{n}}{\raisebox{3.5pt}{$^{(j)}$}}\big)\is{E(\D t)R\eesm_{i}}{R\eesm_{j}}_{\kkc}\\
%&=\sum\nolimits_{i,j=0}^{1}\isb{\int_{\rrb}\big(\sum\nolimits_{m}\lambda_{m}{t^{m}}{\raisebox{3.5pt}{$^{(i)}$}}\big)\big(\sum\nolimits_{n}\sbar\lambda_{n}{t^{n}}{\raisebox{3.5pt}{$^{(j)}$}}\big)E(\D t)R\eesm_{i}}{R\eesm_{j}}_{\kkc}\\
%&\stackrel{{\tt s}}=\int_{\rrb}\big\|\sum\nolimits_{i=0}^{1}\sum\nolimits_{m}\lambda_{m}{t^{m}}{\raisebox{3.5pt}{$^{(i)}$}}E(\D t)R\eesm_{i}\big\|^{2}_{\kkc}\Ge0
%\end{align*}
%where in the passage $\stackrel{\tt s}{=}$ multiplicativity of the spectral integral is used. This establishes \eqref{1.28.09}.
%\end{proof}

\begin{proof}
Note that
\begin{eqnarray*}
\sum\nolimits_{m,n}\lambda_{m}\sbar \lambda_{n}s(m,n)&=&
\sum\nolimits_{m,n}\lambda_{m}\sbar \lambda_{n}\sssm(m,n,\eesm_0,\eesm_0)\\
&+&\sum\nolimits_{m,n}\lambda_{m}\sbar \lambda_{n}m\sssm(m-1,n,\eesm_1,\eesm_0)\\
&+&\sum\nolimits_{m,n}\lambda_{m}\sbar \lambda_{n}n\sssm(m,n-1,\eesm_0,\eesm_1)\\
&+&\sum\nolimits_{m,n}\lambda_{m}\sbar \lambda_{n}mn\sssm(m-1,n-1,\eesm_1,\eesm_1)\\
%&=& \sum\nolimits_{m,n}\sssm(m,n,\lambda_{m}\eesm_0, \lambda_{n}\eesm_0)\\
%&+&\sum\nolimits_{m,n}\sssm(m,n,(m+1)\lambda_{m+1}\eesm_1, \lambda_{n}\eesm_0)\\
%&+&\sum\nolimits_{m,n}\sssm(m,n,\lambda_{m}\eesm_0,(n+1)\lambda_{n+1}\eesm_1)\\
%&+&\sum\nolimits_{m,n} \sssm(m,n,(m+1)\lambda_{m+1}\eesm_1,(n+1)\lambda_{n+1}\eesm_1)\\
&=& \sum\nolimits_{m,n}\sssm(m,n,\lambda_{m}\eesm_0+(m+1)\lambda_{m+1}\eesm_1, \lambda_{n}\eesm_0)\\
&+&\sum\nolimits_{m,n}\sssm(m,n,\lambda_{m}\eesm_0+(m+1)\lambda_{m+1}\eesm_1,(n+1)\lambda_{n+1}\eesm_1)\\
&=& \sum\nolimits_{m,n}\sssm(m,n,\lambda_{n}\eesm_0+(n+1)\lambda_{n+1}\eesm_1,\lambda_{m}\eesm_0+(m+1)\lambda_{m+1}\eesm_1)\\
&\geq& 0,
\end{eqnarray*}
where the last inequality is a consequence of positive definiteness of the form $\sssm$. 
\end{proof}

Now a question appears. Does positive definiteness of $s$ together with the Helton condition  \eqref{phiD} imply existence of {\em positive definite} form $\sssm$ satisfying \eqref{sands}? In the affirmative case this would give an answer to the Sobolev moment problem purely in terms of the original data $s$: a sequence $s$ is a Sobolev moment sequence if and only if it is positive definite and satisfies \eqref{phiD}.

   \bibliographystyle{amsplain}  
   
   %%%%%
   \end{document}